\documentclass[12pt]{article}

\usepackage{amsfonts}
\usepackage{amsmath}
\usepackage{amssymb}
\usepackage[mathscr]{eucal}
\usepackage{graphicx}
\usepackage{fullpage}
\usepackage{times}
\usepackage{multirow}
\usepackage{float}

\def\eod{\vrule height 6pt width 5pt depth 0pt}
\newenvironment{proof}{\noindent {\bf Proof:} \hspace{.2em}}
                      {\hspace*{\fill}{\eod}}

\newcommand{\rowsp}{\mathrm{RowSpace}}
\newcommand{\rrow}{\mathrm{Row}}
\newcommand{\ccol}{\mathrm{Col}}

\newcommand{\LG}{\mathrm{LG}}

\newcommand{\rk}{\mathrm{rank}}

\newcommand{\BB}{\mathfrak{B}}

\newcommand{\DD}{\mathfrak{D}}
\newcommand{\Max}{\mathrm{Max4PC}}
\newcommand{\Min}{\mathrm{Min4PC}}

\newtheorem{theorem}{Theorem}

\newtheorem{corollary}[theorem]{Corollary}

\newtheorem{remark}[theorem]{Remark}
\newtheorem{example}[theorem]{Example}

\newtheorem{lemma}[theorem]{Lemma}

\newcommand{\comment}[1]{}

\def\bpsp{\begin{pspicture}}
\def\epsp{\end{pspicture}}

\def\be{\begin{equation}}
\def\ee{\end{equation}}
\def\ba{\begin{array}}
\def\ea{\end{array}}
\def\bdsc{\begin{description}}
\def\edsc{\end{description}}
\def\bnum{\begin{enumerate}}
\def\enum{\end{enumerate}}
\def\bea{\begin{eqnarray*}}
\def\eea{\end{eqnarray*}}
\def\ben{\begin{eqnarray}}
\def\een{\end{eqnarray}}
\def\btab{\begin{tabular}}
\def\etab{\end{tabular}}
\def\bmat{\begin{bmatrix}}
\def\emat{\end{bmatrix}}

\def\ni{\noindent}

\def\<{\langle}
\def\>{\rangle}

\newcommand\qsD{\raisebox{1.25ex}{$q$}\kern-.65ex \mathfrak{D}}
\newcommand\esD{\raisebox{1.25ex}{$\epsilon$}\kern-.6ex \mathfrak{D}}
\newcommand\qD{\raisebox{1.25ex}{$q$}\kern-.85ex D}
\newcommand\eD{\raisebox{1.25ex}{$\epsilon$}\kern-.6ex D}

\newcommand{\rank}{\mathtt{rank}}
 
\newcommand{\0}{\boldsymbol{0}}

\let\v\relax
\newcommand{\v}{\boldsymbol{v}}
\let\u\relax
\newcommand{\u}{\boldsymbol{u}}
\let\w\relax
\newcommand{\w}{\boldsymbol{w}}

\newcommand{\1}{\mathbf{1}}
\newcommand{\V}{\mathbb{V}}
\newcommand{\U}{\mathbb{U}}

\usepackage{xargs}  
\usepackage[dvipsnames]{xcolor}
\usepackage[colorinlistoftodos,prependcaption,textsize=tiny]{todonotes}
\newcommandx{\unsure}[2][1=]{\todo[linecolor=red,backgroundcolor=red!25,bordercolor=red,#1]{#2}}
\newcommandx{\change}[2][1=]{\todo[linecolor=blue,backgroundcolor=blue!25,bordercolor=blue,#1]{#2}}
\newcommandx{\info}[2][1=]{\todo[linecolor=OliveGreen,backgroundcolor=OliveGreen!25,bordercolor=OliveGreen,#1]{#2}}
\newcommandx{\improvement}[2][1=]{\todo[linecolor=Plum,backgroundcolor=Plum!25,bordercolor=Plum,#1]{#2}}
\newcommandx{\thiswillnotshow}[2][1=]{\todo[disable,#1]{#2}}

\begin{document}

\title{The maximum four point condition matrix of a tree}

\author{}






\author{
        Ali Azimi \\ 
        Department of Mathematics and Applied Mathematics\\
        Xiamen University Malaysia, 43900 \\
        Sepang, Selangor Darul Ehsan, Malaysia \\
        email:  ali.azimi@xmu.edu.my, ali.azimi61@gmail.com
        \and 
        Rakesh Jana \\
        Department of Mathematics \\
        Indian Institute of Technology, Bombay\\
    Mumbai 400 076, India\\
        email: rjana@math.iitb.ac.in
        \and 
        Mukesh Kumar Nagar \\
        Department of Mathematics \\
        Punjab Engineering College\\
        Chandigarh 160 012, India\\
        email: mukesh.kr.nagar@gmail.com
        \and 
                Sivaramakrishnan Sivasubramanian \\
                Department of Mathematics\\
                Indian Institute of Technology, Bombay\\
                Mumbai 400 076, India\\
                email: krishnan@math.iitb.ac.in
                \and 
}

\date{March 17, 2023}

\comment{
}

\maketitle

\begin{abstract}
The Four point condition (4PC henceforth) is a well known 
condition characterising
distances in trees $T$.  Let $w,x,y,z$ be four vertices in $T$ and
let $d_{x,y}$ denote the distance between vertices $x,y$ in $T$.
The 4PC condition says that among the three terms 
$d_{w,x} + d_{y,z}$, $d_{w,y} + d_{x,z}$ and $d_{w,z} + d_{x,y}$ 
the maximum value  equals the second maximum value.

We define an $\binom{n}{2} \times \binom{n}{2}$ sized matrix
$\Max_T$ from a tree $T$ where the rows and columns are 
indexed by size-2 subsets.  The entry of $\Max_T$ corresponding
to the row indexed by $\{w,x\}$ and column $\{y,z\}$ is the 
maximum value among the three terms 
$d_{w,x} + d_{y,z}$, $d_{w,y} + d_{x,z}$ and $d_{w,z} + d_{x,y}$.
In this work, we determine basic properties of this matrix like 
rank, give an algorithm that outputs a family of bases, and 
find the determinant of $\Max_T$ when restricted to our
basis.  We further determine the inertia and the Smith
Normal Form (SNF) of $\Max_T$.
\end{abstract}





\section{Introduction}
\label{sec:intro}

Let $T = (V,E)$ be a tree on $n$ vertices.  Associated to $T$
are several matrices whose entries are functions of distance between the
vertices.  The most well studied of these is the 
$n \times n$ distance matrix $D_T$ of $T$ whose rows and columns are 
indexed by vertices of $T$.  The $(i,j)$-th entry of $D_T$ 
is $d_{i,j}$, 
the distance between vertex $i$ and vertex $j$ in $T$.
About fifty years ago, Graham and Pollak in \cite{graham-pollak-71} 
showed that the determinant of $D_T$ is independent of the 
structure of the tree $T$ and only depends on $n$, the number of 
vertices in $T$.  This result has inspired several 
generalizations (see for example \cite{bapat-jana-pati,
bapat-lal-pati, 
bapat-siva-lapl-tree, bapat-siva-product-dist-sq-dist, 
bapat-siva-second-immanant, bapat-siva-snf-product-dist,
bapat-siva-arithmetic-tutte,
graham-hoffman-hosoya, jana, siva-q-ghh}).  
These papers illustrate the
wealth of results concerning distances in trees.  
We refer the reader to the book \cite{bapat-book} by 
Bapat for a good introduction to such matrices.  
An important condition characterising distances in trees 
was given by Buneman in \cite{buneman-4PC} and is
called the {\it four-point condition} (henceforth denoted
as 4PC).
 

Fix a tree $T$ and denote the distance between vertices
$x,y$ in $T$ as $d_{x,y}$.
The 4PC states that for any four vertices $w,x,y$ and $z$
in $T$, among the three terms 
$d_{w,x} + d_{y,z}$, $d_{w,y} + d_{x,z}$ and $d_{w,z} + d_{x,y}$, 
the maximum value equals the second maximum
value.  In order to understand the 4PC in more
detail, Bapat and Sivasubramanian 
in \cite{bapat-siva-snf-4PC} studied the
$\binom{n}{2} \times \binom{n}{2}$ matrix $M_T$ whose 
rows and columns are indexed by pairs of distinct vertices.  
The entry in the row 
indexed by $\{w,x\}$ and column $\{y,z\}$ of $M_T$ equals the
{\it minimum value} among the three terms
$d_{w,x} + d_{y,z}$, $d_{w,y} + d_{x,z}$ and $d_{w,z} + d_{x,y}$.
They showed the surprising result that the rank of $M_T$ is 
independent of the structure of $T$ and only depends on $n$,
the number of vertices in $T$.  Among other results, they 
also gave the Smith Normal Form (henceforth SNF) of $M_T$. 
It is somewhat surprising that $D_T$, 
the distance matrix of $T$ and $M_T$, the min-4PC matrix
of $T$ have the same rank and the same invariant factors.
We term the matrix $M_T$ as the {\it minimum 4PC matrix} and
also denote it as $\Min_T$.
Analogously, in this work, we define $\Max_T$, the 
$\binom{n}{2} \times \binom{n}{2}$ {\it maximum 4PC matrix} whose
rows and columns are indexed by pairs of distinct vertices.  
The entry in the row 
indexed by $\{w,x\}$ and column $\{y,z\}$ of $\Max_T$ equals the
{\it maximum value} among the three terms
$d_{w,x} + d_{y,z}$, $d_{w,y} + d_{x,z}$ and $d_{w,z} + d_{x,y}$.

Related to this, Azimi and Sivasubramanian in
\cite{aliazimi-siva-steiner-2-dist} studied the 
2-Steiner distance matrix $\DD_2(T)$.  This is also an $\binom{n}{2}
\times \binom{n}{2}$ matrix with the entry in the row indexed by $\{w,x\}$ and
column indexed by $\{y,z\}$ being the number of edges in a minimum 
subtree of $T$
that contains the vertices $w,x,y$ and $z$.  For all positive integers
$k$, one can define $k$-Steiner distance matrices $\DD_k(T)$ and in 
\cite{aliazimi-siva-steiner-2-dist}, the authors show that
when $k=1$, $\DD_1(T) = D_T$ is the usual distance matrix.  Interestingly,
in \cite[Lemma 4]{aliazimi-siva-steiner-2-dist} they
showed that $\DD_2(T) = \frac{1}{2}\Big( \Max_T + \Min_T \Big)$.  Thus,
for any tree $T$, each entry of $\Max_T$ and $\Min_T$ have the 
same parity and their average is the 
corresponding entry of $\DD_2(T)$.

Thus, three $\binom{n}{2} \times \binom{n}{2}$ matrices are 
associated to a tree $T$: 
the maximum 4PC matrix (denoted $\Max_T$), 
the minimum 4PC matrix (denoted $\Min_T$)
and the average 4PC matrix (denoted as $\DD_2(T)$).  
Among these three matrices, results are known for 
two matrices.  See 
Bapat and Sivasubramanian  \cite{bapat-siva-snf-4PC} for results on
$\Min_T$ and see
Azimi and Sivasubramanian \cite{aliazimi-siva-steiner-2-dist} for results
on $\DD_2(T)$.  To the best of our knowledge,
there are no results on the third matrix, $\Max_T$.  
In this paper, we start filling this gap and study $\Max_T$ for 
a tree $T$.  Our first result about $\Max_T$ is the following.

\begin{theorem}
\label{thm:rank}
Let $T$ be a tree on $n\geq 3$ vertices having $p$ pendant vertices.
Then, $$\rk(\Max_T) = 2(n-p).$$
\end{theorem}

For a matrix $M$, let $P,Q$ be subsets of the row and column 
indices respectively.  By $M(P,Q)$ we denote the submatrix of $M$ 
obtained by deleting the rows in $P$ and columns in $Q$.  By $M[P,Q]$
we denote the submatrix of $P$ obtained by restricting $M$ to the
rows in $P$ and the columns in $Q$.

We determine a class of bases $\BB$ of the row space of $\Max_T$ and
for each $B \in \BB$, we 
determine the determinant of the submatrix $\Max_T[B,B]$ 
of $\Max_T$ induced on the rows and columns in $B$.  
Our basis $B$ is constructed using a depth-first search type 
traversal of $T$.  Our algorithm depends on a starting leaf vertex,
and there are further choices as well in the execution of 
our algorithm.  Thus, our output basis $B$ will depend on
these choices and is hence not unique.  Nonetheless, the determinant
of $\Max_T$ when restricted to the rows and columns of 
all such constructed bases has a clean formula which is our 
next result.  


\begin{theorem}
\label{thm:det_value}
Let $B$ be a basis for the row space of $\Max_T$ that is
output by the algorithm described in Lemma \ref{rem:algo}.  Then,
$$\det \Max_T[B,B] 
= (-1)^{n-p}2^{2(n-p-1)}.$$  
\end{theorem}

As mentioned earlier, the invariant factors and hence the 
SNF of $\Min_T$ were found by Bapat and Sivasubramanian in 
\cite[Theorem 2]{bapat-siva-snf-4PC}. 
As a counterpart, in Theorem \ref{thm:snf-max-4pc},
we determine the SNF  of $\Max_T$.   
In \cite[Theorem 18]{aliazimi-siva-steiner-2-dist}, the authors
showed that $\DD_2(T)$ has exactly one positive eigenvalue,
$2n-p-2$ negative eigenvalues and the rest of its eigenvalues
are 0.  If we denote the inertia of a real, symmetric matrix $M$
by the triple $(n_0, n_+, n_-)$, where $n_0$ is the nullity 
of $M$, $n_+$ is the number of positive eigenvalues and 
$n_-$ is the number of negative eigenvalues, then $\DD_2(T)$
has inertia $\Big( \binom{n}{2} - 2n+p+1,1,2n-p-2\Big)$.  
In Theorem \ref{thm:inertia_max4pc}, 
we determine the inertia of $\Max_T$ and show that it has 
$n-p$ positive  eigenvalues and $n-p$ negative eigenvalues.  
Thus Theorem \ref{thm:inertia_max4pc} 
refines Theorem \ref{thm:rank} by
giving the number of positive and negative eigenvalues.

\section{Rank of $\Max_T$}
\label{sec:rank}

Towards proving Theorem \ref{thm:rank}, we start with the following lemmas.
For four vertices $u,v,w,x \in V(T)$, denote by 
$\Max_T(\{u,v\}, \{w,x\})$ the entry of $\Max_T$ indexed by the 
row $\{u,v\}$ and column $\{w,x\}$.  Further, we denote the path
between vertices $u,v$ in $T$ as the {\it $u$-$v$ path}.

\begin{lemma}\label{lem: one-in-rowspace}
Let $T$ be a tree on $n$ vertices.  
Suppose $n$ is a pendant vertex of $T$ with a unique 
neighbour $n-1$.  
Let $u$ be a vertex of $T$ other than $n$ and $n-1$.
Then, for all unordered pairs of distinct vertices $\{i,j\}$, we have  
$$
\Max_T(\{u,n\}, \{i,j\}) = \Max_T(\{u,n-1\}, \{i,j\})+1.
$$
\end{lemma}
\begin{proof}
Recall that $u\neq n-1,n$. Therefore, when $v\neq n$, 
the $v$-$n$ path in $T$ must contain the vertex $n-1$. Thus, we have  
\be \label{eq: e0}
d_{v,n} =d_{v,n-1}+1  
\quad \text{and hence}\quad  
d_{u,n} =d_{u,n-1}+1.
\ee

Let $1\leq i<j\leq n$. Then by the definition of $\Max_T$, we have
\be \label{eq: e1}
\Max_T(\{u,n\}, \{i,j\}) = \max\{ d_{u,n-1}+d_{i,j}+1,\ d_{u,i}+d_{n,j},\ d_{u,j}+d_{n,i}\}.
\ee

\ni We split the proof into two cases with the first case being
when both $i\neq n$ and $j \not= n$.   In this case, by \eqref{eq: e0} it follows that 
\begin{align*}
\Max_T(\{u,n\}, \{i,j\}) &= \max\{ d_{u,n-1}+d_{i,j}+1,\ 
d_{u,i}+d_{n-1,j}+1,\ d_{u,j}+d_{n-1,i}+1
\} \\
&= \Max_T(\{u,n-1\}, \{i,j\}) +1 .
\end{align*}

\ni The second case is when exactly one of $i,j$ equals $n$.  Let $j=n$ and 
hence $i\leq n-1$. By the triangle inequality, 
we have
\be \label{eq: e3}
d_{u,n-1}+d_{i,n-1}\geq  d_{u,i} \quad \text{and}\quad  d_{u,n}+d_{i,n}>  1 + d_{u,i} .
\ee 

\ni Therefore, by \eqref{eq: e3}, we have
$$
\Max_T(\{u,n\}, \{i,n\}) = \max\{ d_{u,n}+d_{i,n},\ d_{u,i}\} =d_{u,n}+d_{i,n}.
$$

\ni Further, note that 
$$\begin{aligned}
\Max_T(\{u,n-1\}, \{i,n\}) & = \max\{ d_{u,n-1}+d_{i,n},\ d_{u,i}+1,\  d_{u,n}+d_{i,n-1}\}\\
&= \max\{
d_{u,n}+d_{i,n}-1,\ d_{u,i}+1,\  d_{u,n}+d_{i,n}-1
\} && [\text{by }\eqref{eq: e0}]\\
&= d_{u,n}+d_{i,n}-1 && [\text{by } \eqref{eq: e0} \text{ and } \eqref{eq: e3}] \\
&=\Max_T(\{u,n\}, \{i,n\}) -1.
\end{aligned}
$$
This completes the proof.
\end{proof}

\begin{lemma}
\label{lem: ell}
Let $T$ be a tree on $n$ vertices.  
Suppose $p,q\in V(T)$ such that $p$ is a pendant vertex of $T$ with $q$ 
being the quasi-pendant vertex adjacent to $p$. Let $u\in V(T)$ be a 
neighbour of $q$  other than $p$ and $B_u$ be the connected 
component of $T-q$ that contains the vertex $u$. Then,
$$
\Max_T(\{p,q\}, \{i,j\}) = \begin{cases}
\Max_T(\{u,q\}, \{i,j\})+2 & \text{ if } i,j\in B_u,\\
\Max_T(\{u,q\}, \{i,j\}) & \text{otherwise}.
\end{cases}$$
\end{lemma}
\begin{proof}
Clearly, for each $i\in T$ and $j\in B_u$, it follows by triangle inequality that 
\be \label{eq: el1}
d_{i,j}< d_{u,i}+d_{q,j}.
\ee

\begin{figure}[h]
\centerline{\includegraphics[scale = 0.9]{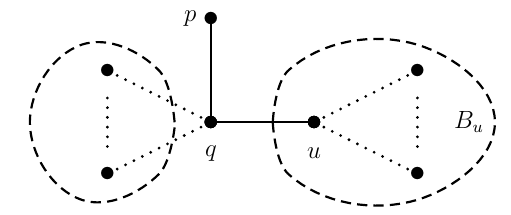}}
\caption{Illustrating Lemma \ref{lem: ell}}
\label{fig:illustr}
\end{figure}


\ni Let us first assume $i,j\in B_u$. 
Clearly $d_{p,v}=d_{u,v}+2$ for each $v\in B_u$. Therefore, it follows that 
\begin{align*}
\Max_T(\{p,q\}, \{i,j\})  & = \max\{d_{p,q}+d_{i,j},\ d_{p,i}+d_{q,j},\ d_{p,j}+d_{q,i}\} \\
& =  \max\{1+d_{i,j},\ d_{u,i}+d_{q,j}+2,\ d_{u,j}+d_{q,i}+2\} \\
& =  \max\{1+d_{i,j}, d_{u,i}+d_{q,j},\ d_{u,j}+d_{q,i}\}+2 && [\text{by } \eqref{eq: el1}]\\
&= \max\{ d_{u,q}+d_{i,j}, \ d_{u,i}+d_{q,j},\ d_{u,j}+d_{q,i}\} +2 \\
&= \Max_T(\{u,q\}, \{i,j\})+2. 
\end{align*}
In the third last line above, we have used the easy to
prove inequality that $1+d_{i,j}$ is smaller than both 
$d_{u,i} + d_{q,j}$ and $d_{u,j} + d_{q,i}$.
We now assume that $i\notin B_u$ and $j\in T$. Note that if $i=p$ 
and $j\in T-p$ then $d_{p,j}+d_{u,q} = d_{q,j}+ d_{p,u}$. It follows that
\begin{align*}
	\Max_T(\{p,q\}, \{p,j\})  & = \max\{d_{p,q}+d_{p,j},\ d_{p,p}+d_{q,j},\ d_{p,j}+d_{p,q}\} \\
	& = \max\{d_{u,q}+d_{p,j},\ d_{p,j}+d_{q,u}\} && [\mbox{as } d_{q,j} < d_{p,j}; \; d_{p,q} = d_{u,q}] \\
	& = \max\{d_{u,q}+d_{p,j},\ d_{p,u}+d_{q,j},\ d_{p,j}+d_{q,u}\} && [\mbox{as } d_{p,u} + d_{q,j} = d_{u,q}+ d_{p,j}]\\
	&= \Max_T(\{u,q\}, \{p,j\}).
\end{align*}

We split the remaining part of the proof into two cases with the
first case being when $i\notin B_u\cup\{p\}$ and $j\in B_u$. 
Clearly, in this case, $d_{i,j}=d_{i,q}+d_{q,u}+d_{u,j}$, and so we get
\be 
d_{u,i}+d_{q,j}=d_{i,j}+1> d_{u,j}+d_{q,i}=d_{i,j}-1. \label{eq: x}
\ee 
 
Therefore, we have 
\begin{align*}
	\Max_T(\{p,q\}, \{i,j\})  
	& =  \max\{1+d_{i,j},\ d_{u,i}+d_{q,j},\ d_{u,j}+2+d_{q,i}\} && [\mbox{as } d_{p,i} =d_{u,i}] \\
	& = \max\{d_{u,q}+d_{i,j},\ d_{u,i}+d_{q,j},\ d_{u,j}+d_{q,i}\}  && [\text{ by } \eqref{eq: x}]\\
	& = \Max_T(\{u,q\}, \{i,j\}).
\end{align*}

\ni  Our second case, is when $i\notin B_u\cup\{p\}$ and $j\notin B_u$.

Note that if $j\neq p$ then $d_{p,i}=d_{u,i}$, $d_{p,j}=d_{u,j}$ and so it follows that 
\begin{align*}
	\Max_T(\{p,q\}, \{i,j\})  
	& = \max\{d_{u,q}+d_{i,j},\ d_{u,i}+d_{q,j},\ d_{u,j}+d_{q,i}\} \\
	& = \Max_T(\{u,q\}, \{i,j\}) 
\end{align*}

Finally, let us assume $j=p$ and so $i\notin B_u\cup\{p\}$.  Clearly, $d_{p,i}=d_{u,i}$. Therefore, we get
\begin{align*}
	\Max_T(\{p,q\}, \{i,p\})  & = \max\{d_{p,q}+d_{i,p},\ d_{p,i}+d_{q,p},\ d_{p,p}+d_{q,i}\} \\
	& = \max\{d_{u,q}+d_{i,p},\ d_{p,i}+d_{q,p}\} && [\mbox{as } d_{q,i} < d_{p,i}] \\
	& = \max\{d_{u,q}+d_{i,p},\ d_{q,i}+d_{u,p}, \ d_{u,i}+d_{q,p}\} 
	\\
	&= \Max_T(\{u,q\}, \{i,p\}).
\end{align*}
This completes the proof.
\end{proof}

%



With the two lemmas above, we are now ready to prove our main 
result of this section.

\begin{proof}(Of Theorem \ref{thm:rank})
We use induction on $n$, the number of vertices in the tree $T$. When 
$n=3$, the only tree is $P_3$, the path on three vertices. It can be 
easily verified that  
$\Max_{P_3}= \left[\begin{array}{rrr}
	2 & 3 & 2 \\
	3 & 4 & 3 \\
	2 & 3 & 2
\end{array}\right]$ and $\rank(\Max_{P_3})=2$. Therefore, the result is true 
for all trees on three vertices. 

Assume that the result is true for all trees on $n-1$ vertices.  
Let $T$ be  a tree on $n$ vertices. 
Without loss of any generality, let $n$ be a pendant vertex that is 
adjacent to $n-1$. 
Let $\widehat T$ be the tree obtained by deleting the vertex $n$ from $T$.  
We divide the proof into two cases based on the degree of vertex $n-1$ in $T$. 

{\bf Case I: There exists a quasi-pendant vertex with degree two.}
We relabel the vertices of $T$ if necessary. We assume that $n$ is 
a leaf of $T$ adjacent to $n-1$ and that $n-1$ has degree 2.
Let $n,n-2$ be the two neighbors 
of $n-1$.   
Let $\widehat T$ be the tree obtained from $T$ by deleting the 
vertex $n$ from $T$.

Let $\V_n$ be the collection of all $2$-size unordered
subsets of $[n]:=\{1,2,\ldots,n\}$ with distinct elements and let 
$\U_{n-1} =\Big\{\{i,n\}:  i\in [n-1]\Big\}$.  We order the elements of $\V_n$ as
$\V_n = \Big(\V_{n-1}, \U_{n-1}\Big)$ and use this order of pairs to
index rows and columns of $\Max_T$. We thus write 
$\Max_T$ in partitioned form as 
 $$
 \Max_T =\begin{bmatrix}
 	\Max_{\widehat T} & \Max_{12}\\
 	\Max_{12}^t & \Max_{22}
 \end{bmatrix},
$$
where $\Max_{12}=\Max_T[\V_{n-1}, \U_{n-1}]$ and $\Max_{22} = 
\Max_T[\U_{n-1}, \U_{n-1}]$.

For a pair $\{u,v\}$ of distinct vertices in $V$, denote 
the row (column) of $\Max_T$ indexed by $\{u,v\}$ as 
$\rrow_{u,v}$ (as $\ccol_{u,v}$ respectively).
We perform the following row and column operations. For 
$1 \leq i < n-1$, 
perform $\rrow_{i,n} = \rrow_{i,n} - 
\rrow_{i,n-1}$ and also perform $\ccol_{i,n} = 
\ccol_{i,n} - \ccol_{i,n-1}$.  If performing row and 
column operations on $M$ gives us the matrix $N$, we 
denote this by $M \sim N$. 
By Lemma \ref{lem: one-in-rowspace}, we get 
$$\begin{aligned}
\Max_T 
&\sim 
\left[ 
\ba{c|c|c} 
\Max_{\widehat T} & 
\ba{cccc}
0 & \cdots& 0 & 1 \\
\vdots & \ddots & \vdots&\vdots\\
0 & \cdots &0 & 1 \\
\ea & 
\u   \\[1.5ex]\hline 
\ba{ccc}
0 & \cdots& 0  \\
\vdots & \ddots & \vdots\\
0 & \cdots &0  \\
1 & \cdots &1  \\
\ea & 
\0  & \ba{c} 0 \\ \vdots \\ 0 \\1  \ea \\[1.5ex]\hline 
\u^t & \ba{cccc}0 & \cdots 0 & 1\ea & 2
\ea\right].
\end{aligned}
$$


Denote the row indexed by $\{u,v\}$ in $\Max_{\widehat T}$ as
$\rrow_{\widehat T}(u,v)$. 
In $\widehat T$, let vertex $n-2$ be adjacent to vertices 
$n-1$ and $n-3$.  Note that we only need the degree of $n-2$ 
in $\widehat T$ to be at least two, not exactly two. Since 
vertex $n-1$ is a pendant vertex in $\widehat T$, by 
Lemma \ref{lem: one-in-rowspace},  for all 
$v\in \widehat T-\{n-1,n-2\}$, we get 
$\rrow_{\widehat T}(n-1, v) = \rrow_{\widehat T}(n-2, v)+\1^t.$

Further, note that $\Max_T(\{n-1,n\}, \{n-1,n-3\})=3$ and 
$\Max_T(\{n-1,n\}, \{n-2,n-3\}) = 4$.
Hence, by performing the row operation 
$\rrow_{n-2, n}=\rrow_{n-2, n}- \rrow_{n-1, n-3}+
\rrow_{n-2, n-3}$ and 
$\ccol_{n-2, n}=\ccol_{n-2, n}-
\ccol_{n-1, n-3}+\ccol_{n-2, n-3}$, we get


 $$\begin{aligned}
	\Max_T &\sim 
	\left[ 
	\ba{c|c|c} 
	\Max_{\widehat T} & 
	\0 &\0   \\\hline 
	\0 & 	\0  & \0\\\hline 
	\0 & \0 & 2 \\\hline 
	\0 & \ba{cccc}\0 & 2\ea & 2
	\ea\right]. 
\end{aligned}$$
This completes the proof of case I. 

{\bf Case II: All quasi-pendant vertices in $T$ have degree at least three:} 
Let $(v_1, \ldots, v_k)$ be a path whose length is equal to the 
diameter of $T$. 
Clearly $v_1$ is a pendant vertex and $v_2$ is a quasi-pendant vertex in $T$. 
As all quasi pendant vertices have degree at least three, $v_2$ 
has another pendant vertex $p$ other than $v_1$ adjacent to it. 
By relabelling, we assume that $v_1 = n$ and $p=n-1$ are two pendant vertices 
in $T$ adjacent to $v_2 = n-2$.  Further, as $n-2$ has degree at least 
three, let $n-3$ be adjacent to $n-2$. 
By Lemma \ref{lem: one-in-rowspace}, 


$$
\rrow_{i,n} = \rrow_{i,n-1}= \rrow_{i,n-2}+\1^t, \qquad \text{for each } i\neq n-2.
$$

Let $B_{n-3}$ be the connected component 
of $T-\{n-2\}$ that contains the vertex $n-3$.  By 
Lemma \ref{lem: ell}, we get

\begin{align*}
	\Max_T(\{n-2,n\},\{i,j\}) &= \begin{cases}
		\Max_T(\{n-3,n-2\},\{i,j\})+2 & \text{ if } i,j\in B_{n-3}\\
		\Max_T(\{n-3,n-2\},\{i,j\}) & \text{otherwise}
	\end{cases}  \\
&=  \Max_T(\{n-2,n-1\}, \{i,j\}),\qquad \text{for each } 1\leq i<j\leq n. 
\end{align*}

Hence, by performing the row operation $\rrow_{i, n}=\rrow_{i, n}-
\rrow_{i, n-2} - \rrow_{n-3, n-1}+\rrow_{n-3, n-2}$ and 
$\ccol_{i, n}=\ccol_{i, n}-\ccol_{i, n-2} - \ccol_{n-3, n-1}+
\ccol_{n-3, n-2}$, when $i\neq n-2$ and  
$\rrow_{n-2, n}=\rrow_{n-2, n} - \rrow_{n-2,n-1}$
and $\ccol_{n-2, n}=\ccol_{n-2, n}-\ccol_{n-2,n-1} $ we get

$$
\Max_T \sim \left[\ba{c|c}
	\Max_{\widehat T} &\0\\[1.5ex] \hline 
\0 & \0
\ea \right].
$$
This completes the proof of case II.  Our proof is complete.
\end{proof}

\section{Smith normal form of $\Max_T$}
\label{sec:snf}

In this section, we determine the invariant factors of 
$\Max_T$.  Our main result is the following.

\begin{theorem}
\label{thm:snf-max-4pc}
Let $T$ be a tree on $n\geq 3$ vertices with $p$ leaves. Then, 
the invariant factors of $\Max_T$ are
$$
\overbrace{0,\cdots,0}^{\binom{n}{2}-2(n-p)}, 1, 1, 
\overbrace{2,\cdots,2}^{2(n-p-1)}.
$$
\end{theorem} 
\begin{proof}
We prove the result by induction on the number of vertices in the tree $T$. 
Our base case is when $n=3$.  In this case, the only tree is the 
path $P_3$ on three vertices.  
Clearly, 
$$
\Max_{P_3} = \bmat 
	2 & 3 & 2 \\
	3 & 4 & 3 \\
	2 & 3 & 2
\emat =
\bmat 
1 & 0 & 0 \\
0 & 1 & 0 \\
1 & 0 & 1
\emat  
\bmat 
1 & 0 & 0 \\
0 & 1 & 0 \\
0 & 0 & 0
\emat
\bmat 
2 & 3 & 2 \\
3 & 4 & 3 \\
0 & 0 & 1
\emat.
$$
Therefore, the result follows when $n=3$. 

We assume that the result is true for all trees on $n-1$ vertices. 
Let $T$ be a tree on $n$ vertices where $n>3$.  Without loss of 
generality, let us assume that $n$ is a pendant vertex adjacent to $n-1$. 
Let $\widehat{T} = T - \{n\}$ be the tree obtained by deleting the
vertex $n$ from $T$. As done earlier, we divide the proof 
into two cases based on the degree of vertex $n-1$ in $T$. 

{\bf Case I:} If the degree of $n-1$ in $T$ is two, then, 
as done in Case I of the proof of Theorem \ref{thm:rank} 
we see that 
$$
\Max_T  \sim 
\left[ 
\ba{c|c|c} 
\Max_{\widehat T} & 
\0 &\0   \\\hline 
\0 & 	\0  & \0\\\hline 
\0 & \0 & 2 \\\hline 
\0 & \ba{cccc}\0 & 2\ea & 2
\ea\right] \sim \left[ 
\ba{c|c|c} 
\Max_{\widehat T} & 
\0 &\0   \\\hline 
\0 & 	\0  & \0\\\hline 
\0 & \ba{cccc}\0 & -2\ea & 0 \\\hline 
\0 & \0  & 2
\ea\right].
$$

The second similarity above is obvious and so our proof is
over in this case.

{\bf Case II:} If the degree of $n-1$ in $T$ is atleast three, then 
as done in Case II of the 
proof of Theorem \ref{thm:rank} 
we see that 
$$
\Max_T  \sim 
\left[\ba{c|c}
\Max_{\widehat T} &\0\\[1.5ex] \hline 
\0 & \0
\ea \right].
$$

Hence, in both cases, the result follows by applying the induction hypothesis. 
\end{proof}

\makeatletter
\def\bbordermatrix#1{\begingroup \m@th
	\@tempdima 4.75\p@
	\setbox\z@\vbox{%
		\def\cr{\crcr\noalign{\kern2\p@\global\let\cr\endline}}%
		\ialign{$##$\hfil\kern2\p@\kern\@tempdima&\thinspace\hfil$##$\hfil
			&&\quad\hfil$##$\hfil\crcr
			\omit\strut\hfil\crcr\noalign{\kern-\baselineskip}%
			#1\crcr\omit\strut\cr}}%
	\setbox\tw@\vbox{\unvcopy\z@\global\setbox\@ne\lastbox}%
	\setbox\tw@\hbox{\unhbox\@ne\unskip\global\setbox\@ne\lastbox}%
	\setbox\tw@\hbox{$\kern\wd\@ne\kern-\@tempdima\left[\kern-\wd\@ne
		\global\setbox\@ne\vbox{\box\@ne\kern2\p@}%
		\vcenter{\kern-\ht\@ne\unvbox\z@\kern-\baselineskip}\,\right]$}%
	\null\;\vbox{\kern\ht\@ne\box\tw@}\endgroup}
\makeatother

\section{Basis for the row space of $\Max_T$}
\label{sec:basis}  

In this section we define a set $\BB$ of bases of the row space 
of $\Max_T$.
We start with the following Corollary about the rank of $\Max_T$ when
we remove a type of leaf from $T$.  
 
\begin{corollary}
\label{cor:remove extra leaf}
Let $T$ be a tree on $n$ vertices with $n>3$. 
Suppose there exist two leaves $u$ and $v$ adjacent to the 
same vertex. Then we have 
 	$$
 	\rk(\Max_T) =\rk(\Max_{T-u})=\rk(\Max_{T-v}).
 	$$
\end{corollary}
\begin{proof}
Follows from Theorem \ref{thm:rank}.
\end{proof}

Let $T$ be a tree on $n$ vertices with $p$ leaves.
By Theorem \ref{thm:rank}, the rank of $\Max_T$ is $2(n-p)$.
To give a basis for the rowspace of $\Max_T$, we need
an index set with cardinality $2(n-p)$. 
We know that the number of blocks in $\LG(T)$, the line graph 
of $T$ is $n-p$. Thus, in order to construct a basis 
for $\rowsp(\Max_T)$ 
we shall take two elements from each block of $\LG(T)$ in the  
following algorithmic way.  Our algorithm is very similar 
to a depth first search (DFS) algorithm.   It turns out,
that our algorithm is easy for non-star graphs and so we 
first handle the case when $T$ is a star tree. 
 
\begin{lemma}\label{lem:star tree}
 Let $T$ be a star tree on $n$ vertices. Then,  the
rank of $\Max_T$ is two. Suppose $1$ is the central vertex of $T$, 
then, the rows indexed by $\{1,i\}$ and $\{j,k\}$ are linearly independent, 
where $1<i\leq j<k \leq n$. Further, let
$\BB$ be the collection $\{\{1,i\},\{j,k\}\}$ where $1<i\leq j<k \leq n$.
Let $B \in \BB$ be a basis.  Then, the determinant of the sub-matrix 
$\Max_T[B,B]$ of $\Max_T$ induced on the rows and columns 
in $B$ is given by 
 $$
 \det \Max_T[B,B] =-1.
 $$
\end{lemma}
\begin{proof}
Let $T$ be a star tree on $n$ vertices and let $1$ be its central 
vertex (having degree $n-1$).   Thus $2,\ldots,n$ are leaves of $T$. 
Let $\V_n$ be the collection of all $2$-size subsets of $[n]$, 
$\V_{1} =\{\{1,i\}\mid 2\leq i\leq n\}$,  
$\V_{2} =\{\{j,k\}\mid 2\leq j<k\leq n\}$. Clearly $\V_n$ can be 
partitioned as  $\V_1\cup \V_2$.  Thus, we write  
$\Max_T$ in partitioned form as 
$$
	\Max_T =\begin{bmatrix}
		2J_1 &3J_2\\
		3J_2^t & 4J_3
	\end{bmatrix},
	$$
where each $J_i$ is an all ones matrix with appropriate size, $i=1,2,3$. 
This completes the proof.
\end{proof}

Note that if $T$ is a tree on three vertices then $T$ is a star tree.
Henceforth, we assume that $T$ is a tree on at least four vertices, 
and that $T$ is not a star tree. 


\begin{remark}\rm 
Let $T$ be a tree on $n\geq 3$ vertices and $\LG(T)$ be its 
line graph. Then, it is easy to see that the number of 
vertices in each block of $\LG(T)$ is at least two. 
\end{remark}

\begin{lemma}[Algorithm to construct a basis for row space of $\Max_T$] 
\label{rem:algo}\rm 
Let $T$ be a tree on $n> 3$ vertices and $\LG(T)$ be its line graph.
Initialise $G=\LG(T)$ and $B=\emptyset$.

Suppose $T$ is not a star tree. Consider a vertex $\{p,q\}$ in $G$ 
where $p$ is a leaf in $T$ and $q$ is adjacent to $p$.
Set the vertex $\{p,q\}$ of $\LG(T)$ as 
a {\em starting vertex} and set the {\em next starting vertex set} 
as the empty set.

\begin{itemize}
\item[Step 1.] Note that the {\em starting vertex} cannot be a cut 
vertex of $G$. Therefore, there exists a unique block $B_c$ in 
$G$ that contains the {\em starting vertex}. We call the 
block $B_c$ as the {\em current block}.

\item[Step 2.] If the current block $B_c$ contains a cut vertex of $G$. 
\begin{itemize}
		\item[a.] We choose a cut vertex $\{u,v\}$ in $B_c$ and call it 
the {\em chosen vertex}.  Further, add all other cut vertices of $G$ 
that are in $B_c$ (that is, other than the cut vertex $\{u,v\}$) 
into the {\em next starting vertex set}. 
		 
		\item[b.] Let $\widehat G$ be the graph obtained from $G$ by 
removing all edges of $B_c$ and then deleting all the non-cut 
vertices of $B_c$ 
from $G$. 
(Define $G = G - \{ \mbox{edges in $B_c$ } \}$ 
and then define $\widehat G = G - (B_c - 
\{\mbox{non cut vertices in $B_c$}\})$. 
Thus,
$\widehat G$ has one block lesser than  $G$.)  
Note that all the cut vertices of $G$ which are in $B_c$ 
become non-cut vertices in $\widehat G$.  Set $G=\widehat G$.  
		
\item[c.] To our set $B$, we add two elements; the {\em starting vertex} 
and the symmetric difference between the starting vertex and the 
chosen vertex $\{u,v\}$. 
		
		\item[d.]  Redefine the {\em starting vertex} as the 
{\em chosen vertex} $\{u,v\}$ and go to Step 1.  
	\end{itemize}

\item[Step 3.] If the current block $B_c$ does not contain any 
cut vertex of $G$. 
\begin{itemize}
	\item[a.] Choose a vertex $\{u,v\}$ in $B_c$ other than the 
starting vertex and call it the {\em chosen vertex}. 
		
		\item[b.] Add the two elements {\em starting vertex} 
and the {\em chosen vertex $\{u,v\}$} to $B$.
		
		\item[c.] 
Define $\widehat G = G - B_c$.
Set $G=\widehat G$.  
		
\item[d.] If {\em next starting vertex set} is the empty set, 
output $B$ and terminate 
the algorithm. Otherwise, choose an element, say $\{w,x\}$ from the 
{\em next starting vertex set}, and delete it from  
{\em next starting vertex set}.  Now redefine the {\em starting vertex} 
as $\{w,x\}$ and go to Step 1.
\end{itemize} 
\end{itemize}
\end{lemma}

In the following example we illustrate the algorithm described in Lemma
 \ref{rem:algo}. 

\begin{example}\rm \label{ex: basis}
Consider the tree $T$ shown below. Its line graph $\LG(T)$ 
is shown on the right. 
\begin{center}
\begin{minipage}{0.45\textwidth}
	\includegraphics[scale=1.25]{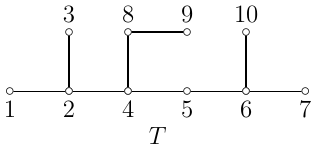}
\end{minipage}
\begin{minipage}{0.45\textwidth}
\includegraphics{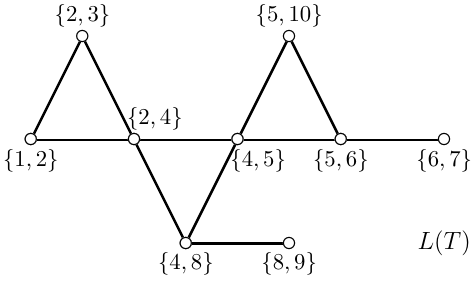}
\end{minipage} 
\end{center} 

Suppose we start the algorithm by choosing the leaf $1$ in $T$. 
Therefore, $\{1,2\}$ is our {\em starting vertex}.  We use black 
colored, gray colored, and red colored nodes to represent the 
\textit{starting vertex}, the \textit{chosen vertex}, and 
the \textit{next starting vertex} respectively.  

\begin{center}
	\begin{minipage}{0.47\textwidth}
		\begin{flushleft}
			\includegraphics[scale=1]{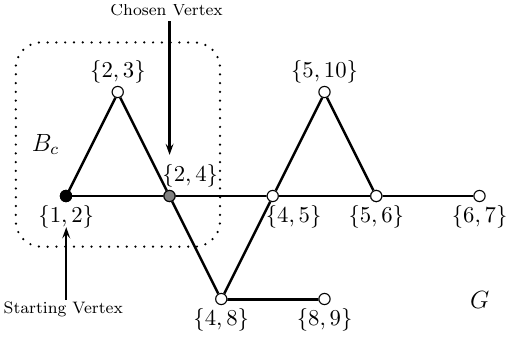}
		\end{flushleft}
	\end{minipage}
	\begin{minipage}{0.47\textwidth}
		\begin{flushright}
			\includegraphics[scale=1]{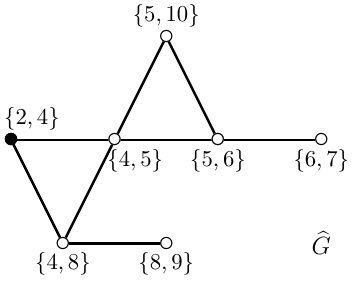}
		\end{flushright}
	\end{minipage}
\end{center} 

Recall that initially $G=\LG(T)$.  The block containing $\{1,2\}$ 
is the current block $B_c$ and is marked using dotted lines.
Clearly, $B_c$ contains only one other cut vertex of $G$ (vertex 
$\{2,4\}$) and so the {\em chosen vertex} is $\{2,4\}$.  
As there is only one cut vertex of $G$ in $B_c$, by Step 2a, the 
{\em next starting vertex set} is the empty set
(see the graph drawn on the left in the above diagram).  By Step 2b, 
construct $\widehat G$ from $G$ by deleting all edges of $B_c$ along 
with vertices $\{1,2\}$ and $\{2,3\}$. (See the graph drawn on the 
right in the above diagram.) By Step 2c, add
$\{1,2\}$ and $\{1,4\}$ (the symmetric difference of $\{1,2\}$ and 
$\{2,4\}$) to $B$. By Step 2d, we make $\{2,4\}$ as the 
current {\em starting vertex} and proceed to Step 1.

\begin{center}
	\begin{minipage}{0.47\textwidth}
		\begin{flushleft}
			\includegraphics[scale=1]{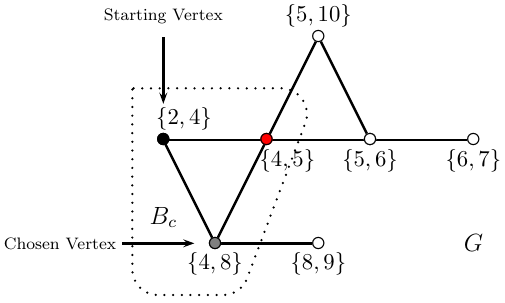}
		\end{flushleft}
	\end{minipage}
	\begin{minipage}{0.47\textwidth}
		\begin{flushright}
			\includegraphics[scale=1]{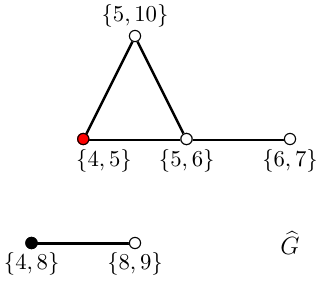}
		\end{flushright}
	\end{minipage}
\end{center} 

As the starting vertex is $\{2,4\}$, the block that contains it is
$B_c$.
Note that $B_c$ contains two cut vertices of $G$: viz 
$\{4,5\}$ and $\{4,8\}$. By Step 2a, we choose $\{4,8\}$ as our 
{\em chosen vertex} and so the {\em next starting set} is $\{\{4,5\}\}$. 
(See the left graph in the above diagram.) Construct $\widehat G$ from $G$ 
by performing Step 2b.  $\widehat G$ is shown in the graph on the right, in
the above diagram. By Step 2c, after adding  $\{2,4\}$ and $\{2,8\}$ 
(the symmetric difference of $\{2,4\}$ and $\{4,8\}$) to  $B$, 
the set $B$ becomes $B= \{ \{1,2\}, \{1,4\}, \{2,4\}, \{2,8\}\}$. 
By Step 2d, we make $\{4,8\}$ as the current {\em starting vertex} 
and proceed to Step 1 again.

As the starting vertex is $\{4,8\}$, the current block $B_c$ is 
the one containing it and is drawn with dotted lines. 
Note that $B_c$ does not contain any cut vertex of $G$. By Step 3a, 
we choose $\{8,9\}$ as our {\em chosen vertex}. (See the left 
graph in the below diagram.) By applying Step 3b, the set $B$ now 
becomes 
$B= \{ \{1,2\}, \{1,4\}, \{2,4\}, \{2,8\}, \{4,8\}, \{8,9\}\}$.  Note 
that no symmetric difference is performed to the newly added
elements of $B$ at this stage.
Construct $\widehat G$ from $G$ by following Step 3c, 
which is shown in the right graph of the below diagram. Note 
that $\{4,5\}$ is the only element on {\em next starting vertex set}. 
Thus, $\{4,5\}$ is our {\em starting vertex} and we proceed to Step 1.

\begin{center}
	\begin{minipage}{0.47\textwidth}
		\begin{flushleft}
			\includegraphics[scale=1]{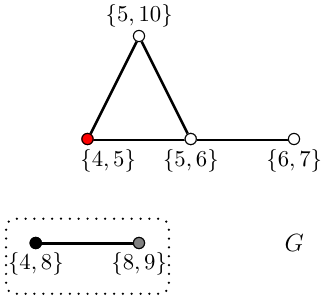}
		\end{flushleft}
	\end{minipage}
	\begin{minipage}{0.47\textwidth}
		\begin{flushright}
			\includegraphics[scale=1]{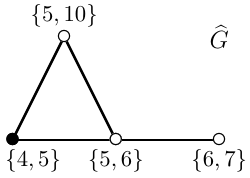}
		\end{flushright}
	\end{minipage}
\end{center} 

Since  $\{4,5\}$ is our starting vertex, by Step 1, 
the current block $B_c$ is $\{\{4,5\},\{5,6\},\{5,10\}\}$. Note that $B_c$ 
contains only one cut vertex of $G$ which is our chosen vertex and  is
marked with gray colored node in the below figure. By Step 2b, we construct the graph $\widehat G$, see the right side graph in the below figure. After applying Step 2c, the set $B$ becomes 
$$
B= \{ \{1,2\}, \{1,4\}, \{2,4\}, \{2,8\}, \{4,8\}, \{8,9\}, \{4,5\},\{4,6\}\}.
$$ 
Now proceed to Step 1 again with $\{5,6\}$ as our {\em starting vertex}. 
\begin{center}
	\begin{minipage}{0.47\textwidth}
		\begin{flushleft}
			\includegraphics[scale=1]{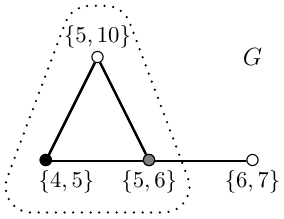}
		\end{flushleft}
	\end{minipage}
	\begin{minipage}{0.47\textwidth}
		\begin{flushright}
			\includegraphics[scale=1]{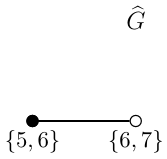}
		\end{flushright}
	\end{minipage}
\end{center} 

Since  $\{5,6\}$ is starting vertex, by Step 1, the current block 
$B_c$ is $\{\{5,6\},\{6,7\}\}$. Note that $B_c$ does not contain 
any cut vertex of $G$.  After applying Steps 3a-b, the set $B$ becomes 
$$
B= \{ \{1,2\}, \{1,4\}, \{2,4\}, \{2,8\}, \{4,8\}, \{8,9\}, \{4,5\},\{4,6\}, \{5,6\},\{6,7\}\}.
$$ 
Note that by applying Step 3c, we will get an empty graph as 
$\widehat G$. Since there is no  element in {\em  next starting 
vertex set}, by Step 3d, we terminate the process. Note that the 
final set $B$ contains $10$ elements and each block of $\LG(T)$ 
contributes exactly two elements to $B$. We mark 
those elements of $\LG(T)$ using red color in Figure \ref{fig: ex-basis}.
Note that red coloured edges of $\LG(T)$ mean we take the
symmetric difference of the end points of this edge to get a 
2-sized subset of $V(T)$.

\begin{figure}[t]\centering
	\includegraphics{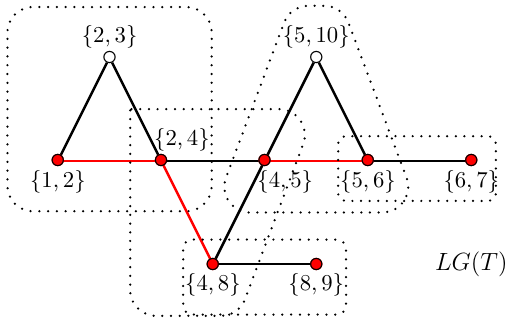}
	\caption{Red colored elements of $\LG(T)$ indicate elements 
contributed to $B$ in Example \ref{ex: basis}.}\label{fig: ex-basis}
\end{figure} 
\end{example}

Define $\BB$ to be the union of all the sets $B$ obtained by 
the algorithm described in Lemma \ref{rem:algo} where the 
union is taken over all possible choices of starting vertices.
In our next result, we discuss some properties of the output 
$B$ obtained by applying the algorithm.

\begin{theorem}\label{th: det}
Let $T$ be a tree  on $n\geq  3$ vertices with $p$ leaves. 
Let $B \in \BB$ be an output of our algorithm described in 
Lemma \ref{rem:algo}. Then, the following is true. 

\begin{itemize}
\item[a.] If $T$ is not a star tree, then there exist unique 
vertices  $u,v,w\in T$ such that $\{u,v\}, \{u,w\} \in B$ 
with $d(u)=1$, $d(w)> 1$ (recall $d(u)$ is the degree of vertex 
$u$) and with both $\{u,v\}, \{v,w\} \in E(T)$.
	
	\item[b.] The number of elements in $B$ is $2(n-p)$.
	
	\item[c.] The set $B$ is a basis for the row space of $\Max_T$.
	
	\item[d.] 
We have $\displaystyle \det 
\Max_T[B,B] =(-1)^{n-p}\ 2^{ 2(n-p-1)}.$
\end{itemize}	
\end{theorem}

\begin{proof}{\em Proof of Item a.}  In the algorithm, the initial 
starting vertex $\{u,v\}$ is clearly taken with $u$ being a 
leaf adjacent to $v$. Since $T$ is not a star, the number of blocks 
in  $\LG(T)$ is at least two.  Thus, the block of $\LG(T)$ that 
contains the vertex $\{u,v\}$ must contain a cut vertex of 
$G$. By Step 2 of Lemma \ref{rem:algo}, it follows that 
there exists a cut vertex $\{v,w\}\in \LG(T)$ that give rise to 
$\{u,v\}, \{u,w\}\in B$.

We now show the uniqueness of $u$.  Suppose, to the contrary,
there are  $u_1,v_1,w_1\in T$ such that 
$\{u_1,v_1\}, \{u_1,w_1\}\in B$ with $u \neq u_1$, 
$\{u_1, v_1\}$ and $\{v_1, w_1\} \in E(T)$, with  
$d(u_1)=1$, and $d(w_1)>1$. 
Clearly, both $\{u_1,v_1\}$, $\{u_1,w_1\}$ were added 
to $B$ in Step 2c. Since $\{u_1,w_1\}$ is not an edge in $T$, it 
follows that in some step  $\{u_1,v_1\}$ was a starting vertex. 
As  $u\neq u_1$, it follows that degree of the vertex $u_1$ in 
$T$ is at least two. This contradicts that $u_1$ is a leaf.

{\em Proof of Item b.} If $T$ is a star tree, then, 
there is nothing to prove. Suppose  $T$ is not a star tree. 
By Lemma \ref{rem:algo}, note that in each step, exactly one block 
of the line graph of $T$ is removed and exactly two elements 
corresponding to that block are added in $B$. Since the 
number of block in $\LG(T)$ is $(n-p)$, 
$B$ has $2(n-p)$ elements.

{\em Proof of Item c.} We use induction on $n$. Our base case 
when $T$ has three vertices can easily be verified.
Suppose the result is true for all tree on $n-1$ vertices. Let 
$T$ be a tree on $n$ vertices.  If $T$ is a star tree then 
the result follows by Lemma \ref{lem:star tree}.

Suppose $T$ is not a star tree. Let $B$ be a set output
by our algorithm described in Lemma \ref{rem:algo}. 
By part (a), there exist unique vertices $u,v,w$ in $B$ 
such that $\{u,v\},\{u,w\}\in T$ with $d(u)=1$ and 
$\{u,v\} , \{v,w\} \in E(T)$. Without of loss of 
generality let us assume $u=1, v=2$, and $w=3$. 
Note that, by Lemma \ref{lem: one-in-rowspace}, we get
\begin{equation}\label{eq:allone}
\Max_{T}[B,\{1,3\}]=\Max_{T}[B,\{2,3\}]+\1.  
\end{equation}

Now note that for each leaf $l\neq 1$ in $T$, if $\{l,v\}\in B$ 
for some $v$ then $\{v,l\} \in E(T)$ and $\{v,w\}\in B$, where $w$ is a 
neighbor of $v$ other than $l$. 
Without loss of any generality, let us assume $x$ is a leaf lying
on a path whose length is the diameter of $T$ with $\{x, y\} \in E(T)$.
Note that if there is more than one leaf attached at $y$, then 
by the induction hypothesis and Corollary \ref{cor:remove extra leaf}, 
it follows that $B$ is a basis for the row space of $\Max_T$. 

Let us assume $d(y)=2$ and $\{y,z\} \in E(T)$.  It follows that 
$\{x,y\}, \{y,z\} \in B$.
Suppose $w$ be the neighbor of $z$ other than $y$ such that $\{z,w\}\in B$. 
Let $\widehat{B}$ be the set obtained by applying 
Lemma \ref{rem:algo} on $T-x$ in the same sequence as it was 
applied for $T$ while obtaining the set $B$. By induction hypothesis, 
$\widehat{B}$ is a basis for the row space of $\Max_{T-x}$. 
Now we divide the remaining part of the proof into two cases.

We first assume that $\{y,z\}\in\widehat{B}$. It follows that 
$\{y,w\}\in B$.  Then, $\widehat B=B\setminus\{\{x,y\},\{y,w\}\}$.  
Note that the matrix $\Max_{T}[B,B]$ can be partitioned as 
$$
\Max_{T}[B,B]= 
\bbordermatrix{
&& \{x,y\} & \{y,w\}\cr
& \Max_{T-x}[\widehat B,\widehat B] & \u & \v \cr 
\{x,y\}& \u^t & 2 & 3\cr
\{y,w\}& \v^t & 3 & 4
},
$$
where  $\u=\Max_{T}[\widehat B,\{x,y\}]$ and 
$\v=\Max_{T-x}[\widehat B,\{y,w\}]$.

By Lemma \ref{lem: one-in-rowspace}, we have
$$\Max_{T-x}[\widehat B,\{y,w\}]=\Max_{T-x}[\widehat B,\{z,w\}]+\1.$$  
Further, note that $\Max_{T}(\{z,w\},\{x,y\}) =4$ and  
$\Max_{T}(\{z,w\},\{y,w\}) =3$. Hence, by performing the
row operation $\rrow_{y,w} = \rrow_{y,w}-
\rrow_{z,w} - (\rrow_{1,3}-\rrow_{2,3})$ 
in $\Max_T[B,B]$ and an identical column operation, we obtain 
$$
\bbordermatrix{
	&& \{x,y\} & \{y,w\}\cr
	& \Max_{T-x}[\widehat B,\widehat B] & \u & \0 \cr 
	\{x,y\}& \u^t & 2 & -2\cr
	\{y,w\}& \0^t & -2 & 0
}.
$$

It follows that $\det \Max_{T}[B,B] = (-4)\times
\det \Max_{T-x}[\widehat B,\widehat B]$. Hence, by induction 
hypothesis, it follows that $B$ is a basis for the row space 
of $\Max_T$.

Now we consider the case when $\{y,z\}\notin\widehat B$.  
It follows that $\widehat B= B\setminus\{\{y,z\},\{x,y\}\}$. We can 
clearly partition the matrix $\Max_{T}[B,B]$ as follows.
$$
\Max_{T}[B,B]= 
\bbordermatrix{
	&& \{y,z\} & \{x,y\}\cr
	& \Max_{T-x}[\widehat B,\widehat B] &\w& \u \cr 
	\{y,z\}& \w^t & 2 & 2\cr
	\{x,y\}& \u^t & 2 & 2
},
$$
where  $\u=\Max_{T-x}[\widehat B,\{x,y\}]$ and  
$\w=\Max_{T}[\widehat B,\{y,z\}]$.

By Lemma  \ref{lem: ell}, it follows that $\u=\w+2\1$. 
Hence, by performing the row operation 
$\rrow_{x,y} = \rrow_{x,y}-\rrow_{y,z} - (\rrow_{1,3}-\rrow_{2,3})$ 
in $\Max_{T}[B,B]$ and an identical column operation, we get
$$
\bbordermatrix{
	&& \{y,z\} & \{x,y\}\cr
	& \Max_{T-x}[\widehat B,\widehat B] & \u & \0 \cr 
	\{y,z\}& \u^t & 2 & -2\cr
	\{x,y\}& \0^t & -2 & 0
}
$$
It follows that $\det \Max_{T}[B,B] = (-4) \times 
\det \Max_{T-x}[\widehat B,\widehat B]$. Hence, by the 
induction hypothesis, it follows that $B$ is a basis 
for the row space of $ \Max_{T}$, completing the proof.

{\em Proof of Item d.} 
The result follows from the proof of Item c and noting 
that when $n=3$, the determinant value is $(-1)$. 
\end{proof}

%
%
%
%
%

\section{Inertia of $\Max_T$}
\label{sec:inertia}
In this section, we determine the inertia of $\Max_T$. 
For an $n\times n$ real symmetric matrix $A$, we denote its 
number of positive, negative and zero eigenvalues by $n_+$, 
$n_-$ and $n_0$, respectively.  We denote the inertia of 
$A$ by $\text{Inertia}(A)$ and define it as the 
triple $(n_0, n_+, n_-)$. Since $A$ is a real symmetric matrix, 
$n_0+ n_+ + n_-=n$.  We recall the well known Sylvester's law of inertia. 

\begin{theorem}[Sylvester's Law of Inertia]
\label{th syl-iner}
Let $A$ be a real symmetric matrix of order $n$ and let $Q$ be a
nonsingular matrix of order $n$. Then $\textrm{Inertia}(A) = \textrm{Inertia}(QAQ^t)$.
\end{theorem}

The main result of this Section is the following where we determine
the inertia of $\Max_T$.

\begin{theorem}
\label{thm:inertia_max4pc}
Let $T$ be a tree on $n$ vertices with $p$ leaves. Then, the
inertia of $\Max_T$ is
$$
\text{Inertia}(\Max_T) = (n_0,n_+,n_-) = 
\left(\binom{n}{2}-2(n-p),\: n-p,\: n-p\right).$$
\end{theorem}

\begin{proof}
By induction on $n$, we first prove that if $B$ is a basis for the row space of $\Max_T$ obtained by applying Lemma \ref{rem:algo} then
$\text{Inertia}(\Max_T[B,B]) = (0,n-p,n-p)$. 

If $T$ is tree on $n<4$ vertices then the result can be verified easily. 
Now notice that if  two leaves $u$ and $v$ of $T$ have a common neighbor,
then by Corollary \ref{cor:remove extra leaf}, we have
$$
\rk(\Max_T) =\rk(\Max_{T-u})=\rk(\Max_{T-v}).
$$
Hence, the result follows by applying induction hypothesis on the tree 
$T-u$. We thus assume that $T$ is a tree such that every 
quasi-pendant vertex of $T$ is adjacent to exactly one leaf.  
Let $B$ be a basis of the row space of $\Max_T$ 
obtained by applying Lemma \ref{rem:algo}.

Without loss of generality, assume that $n$ is a leaf adjacent to 
$n-1$ with  $\{n,n-1\}\in B$  but $\{n,n-2\}\notin B$ where  
$n-2$ is a neighbour of $n-1$. 
We first compute $\text{Inertia}(\Max_T[B, B])$.
The proof of Item c of Theorem \ref{th: det} gives
$\det \Max_T[B,B]=-4 \det\Max_{T-n}[\widehat B,\widehat B]$, 
where $\widehat B$ is the basis for the row space of $\Max_{T-n}$ 
obtained by applying Lemma \ref{rem:algo} on $T-n$ in 
the same sequence as applied to get $B$.
	
Clearly, by Theorem \ref{thm:rank}, $\rk(\Max_T)=\rk(\Max_{T-n})+2$ 
and so the number of nonzero eigenvalues of $\Max_{T}[B, B]$ is 
two more than of $\Max_{T-n}[\widehat B, \widehat B]$.  Since 
the product of	$\det \Max_T[B,B]$ and 
$\det\Max_{T-n}[\widehat B,\widehat B]$ is negative, the 
number of positive eigenvalues of $\Max_{T}[B, B]$ is exactly one 
more than that of $\Max_{T-n}[\widehat B,\widehat B]$.  
This argument also gives the result on the number 
of negative eigenvalues of  $\Max_{T}[B, B]$.
Hence, 
by the induction hypothesis, $\text{Inertia}(\Max_T[B,B])= (0,n-p,n-p)$.
	
Since $\Max_{T}$ is a real symmetric matrix, there exists an 
orthogonal matrix $Q$ such that $Q\Max_T Q^t = \begin{bmatrix}
\Max_T[B,B] & \0 \\\0&\0
\end{bmatrix}$. Hence, the result holds by applying Theorem \ref{th syl-iner}. 
\end{proof}

In our final result, we explicitly describe the eigenvalues 
of $\Max_T$ when $T$ is a star tree. 

\begin{theorem}
Let $S_n$ be the star tree on $n$ vertices. Then, we have
$$
\det(xI-\Max_{S_n}) = x^{\binom{n}{2}-2}\left(x^2-2(n-1)^2x- (n-1)\binom{n-1}{2}\right),
$$
and the nonzero eigenvalues of $\Max_{S_n}$ are
$$
(n-1)^2 \pm \sqrt{(n-1)^4+ (n-1)\binom{n-1}{2}}.
$$
\end{theorem}

\begin{proof}
Clearly, $\rk(\Max_{S_n})=2$. Let $\lambda$ and $\mu$ be the
two nonzero eigenvalues of $\Max_{S_n}$. Now note that
$$
\Max_{S_n} = \begin{bmatrix}
2J_{(n-1)\times (n-1)} & 3J_{(n-1)\times \binom{n-1}{2}}\\
3J_{\binom{n-1}{2}\times (n-1)} & 4J_{\binom{n-1}{2}\times \binom{n-1}{2}}
\end{bmatrix}.
$$
Therefore, $\lambda+\mu = 2(n-1)+4\binom{n-1}{2} = 2(n-1)^2$. Further, 
note that the sum of all $2\times 2$  principal minors of $\Max_{S_n}$ is
$
- (n-1)\binom{n-1}{2}
$. It follows that
$$
\lambda\mu = -(n-1)\binom{n-1}{2}.
$$	
Solving the quadratic gives us the two individual roots.
Further, the characteristic polynomial of $\Max_{S_n}$ is given by
$$
x^{\binom{n}{2}-2}\left(x^2-2(n-1)^2x- (n-1)\binom{n-1}{2}\right).
$$
This completes the proof.
\end{proof}

\comment{
\begin{theorem}
Let $T$ be the star tree on $n$ vertices. Then, we have 
$$
\det(xI-\Max_T) = x^{\binom{n}{2}-2}\left(x^2-2(n-1)^2x-	(n-1)-\binom{n-1}{2}\right),
	$$
	and thus the nonzero eigenvalues of $\Max_T$ are 
	$$
	(n-1)^2 \pm \sqrt{(n-1)^4+ (n-1)\binom{n-1}{2}}.
	$$
\end{theorem}

\begin{proof}
	Clearly, $\rk(\Max_T)=2$. Let $\lambda$ and $\mu$ be two nonzero eigenvalues of $\Max_T$. Now note that 
	$$
	\Max_T = \begin{bmatrix}
		2J_{(n-1)\times (n-1)} & 	3J_{(n-1)\times \binom{n-1}{2}}\\
		3J_{\binom{n-1}{2}\times (n-1)} & 	4J_{\binom{n-1}{2}\times \binom{n-1}{2}}
	\end{bmatrix}.
	$$
	Therefore, $\lambda+\mu = 2(n-1)+2(n-1)(n-2) = 2(n-1)^2$. Further, note that the sum all  $2\times 2$  principal minors of $\Max_T$ is 
	$
	- (n-1)+\binom{n-1}{2}
	$. It follows that 
	$$
	\lambda\mu = -(n-1)-\binom{n-1}{2}.
	$$
	Therefore, the characteristic polynomial of $\Max_T$ is given by 
	$$
	x^{\binom{n}{2}-2}\left(x^2-2(n-1)^2x-	(n-1)-\binom{n-1}{2}\right).
	$$ 
	This completes the proof.
\end{proof}
}


\begin{thebibliography}{10}
	
	\bibitem{aliazimi-siva-steiner-2-dist}
	{\sc Azimi, A., and Sivasubramanian, S.}
	\newblock The 2-steiner distance matrix of a tree.
	\newblock {\em Linear Algebra and its Applications 655\/} (2022), 65--86.



\bibitem{bapat-book}
   {\sc Bapat, R.~B.}
   \newblock {G}raphs and {M}atrices.
   \newblock {\em Hindustan Book Agency} (2014), Second edition.
	

\bibitem{bapat-jana-pati}
{\sc Bapat, R.~B., Jana, R. and Pati, S.}
\newblock The bipartite distance matrix of a nonsingular tree.
\newblock {\em Linear Algebra and its Applications 631\/} (2021), 254-281


	\bibitem{bapat-lal-pati}
	{\sc Bapat, R.~B., Lal, A.~K., and Pati, S.}
	\newblock A $q$-analogue of the distance matrix of a tree.
	
	\bibitem{bapat-siva-lapl-tree}
	{\sc Bapat, R.~B., and Sivasubramanian, S.}
	\newblock Identities for minors of the {L}aplacian, resistance and distance
	matrices.
	\newblock {\em Linear Algebra and its Applications 435\/} (2011), 1479--1489.
	
	\bibitem{bapat-siva-product-dist-sq-dist}
	{\sc Bapat, R.~B., and Sivasubramanian, S.}
	\newblock {P}roduct {D}istance {M}atrix of a {G}raph and {S}quared {D}istance
	{M}atrix of a {T}ree.
	\newblock {\em Applicable Analysis and Discrete Mathematics 7\/} (2013),
	285--301.
	

	\bibitem{bapat-siva-second-immanant}
	{\sc Bapat, R.~B., and Sivasubramanian, S.}
	\newblock The {S}econd {I}mmanant of some {C}ombinatorial {M}atrices.
	\newblock {\em Transactions on Combinatorics 4, (2)\/} (2015), 23--35.
	
	\bibitem{bapat-siva-snf-product-dist}
	{\sc Bapat, R.~B., and Sivasubramanian, S.}
	\newblock The {S}mith normal form of product distance matrices.
	\newblock {\em Special Matrices 4\/} (2016), 46--55.
	
	\bibitem{bapat-siva-arithmetic-tutte}
	{\sc Bapat, R.~B., and Sivasubramanian, S.}
	\newblock The {A}rithmetic {T}utte polynomial of two matrices associated to
	{T}rees.
	\newblock {\em Special Matrices 6\/} (2018), 310--322.
	
	\bibitem{bapat-siva-snf-4PC}
	{\sc Bapat, R.~B., and Sivasubramanian, S.}
	\newblock Smith {N}ormal {F}orm of a distance matrix inspired by the four-point
	condition.
	\newblock {\em Linear Algebra and its Applications 603\/} (2020), 301--312.
	
	\bibitem{buneman-4PC}
	{\sc Buneman, P.}
	\newblock A {N}ote on the {M}etric {P}roperties of {T}rees.
	\newblock {\em Journal of Combin Theory (B) 17\/} (1974), 48--50.
	
	\bibitem{graham-hoffman-hosoya}
	{\sc Graham, R.~L., Hoffman, A.~J., and Hosoya, H.}
	\newblock On the distance matrix of a directed graph.
	\newblock {\em Journal of Graph Theory 1\/} (1977), 85--88.
	
	\bibitem{graham-pollak-71}
	{\sc Graham, R.~L., and Pollak, H.~O.}
	\newblock On the addressing problem for loop switching.
	\newblock {\em Bell System Tech. J 50\/} (1971), 2495--2519.
	
\bibitem{jana}
{\sc Jana, R.}
\newblock A $q$-analogue of the bipartite distance matrix of a nonsingular tree.
\newblock {\em Discrete Mathematics 346(1)\/} (2023), 113153

	\bibitem{siva-q-ghh}
	{\sc Sivasubramanian, S.}
	\newblock A $q$-analogue of {G}raham, {H}offman and {H}osoya's theorem.
	\newblock {\em Electronic Journal of Combinatorics 17(1)\/} (2010), N21.
	
\end{thebibliography}
\end{document}